\documentclass[12pt]{extarticle}
\usepackage{amsmath, amsthm, amssymb, mathtools, hyperref, color}
\usepackage[shortlabels]{enumitem}
\usepackage{graphicx}
\usepackage{adjustbox}
\usepackage{xspace}
\usepackage{cleveref}
\usepackage{tikz}
\usepackage{multirow}
\usepackage{booktabs}
\usepackage{array}
\usepackage{tabularx}
\usepackage{fancyvrb}
\usetikzlibrary{arrows,calc,fit,matrix,positioning,cd}
\usepackage{diagbox}
\usepackage{comment}
\usepackage{float}

\definecolor{cb-yellow}{RGB}{221,170,51}
\definecolor{cb-red} {RGB}{187,85,102}
\definecolor{cb-green}{RGB}{17,119,51}

\hypersetup{
    colorlinks=true,
    linkcolor=cb-red,
    filecolor=cb-red,      
    urlcolor=cb-green,
    citecolor=cb-green,
}

\oddsidemargin \evensidemargin
\newtheorem{theorem}{Theorem}
\newtheorem*{theorem*}{Theorem}
\numberwithin{theorem}{section}
\newtheorem{proposition}[theorem]{Proposition}
\newtheorem{lemma}[theorem]{Lemma}
\newtheorem{corollary}[theorem]{Corollary}
\newtheorem{conjecture}[theorem]{Conjecture}
\newtheorem{question}[theorem]{Question}

\theoremstyle{definition}
\newtheorem{definition}[theorem]{Definition}
\newtheorem{remark}[theorem]{Remark}
\newtheorem{example}[theorem]{Example}

\newcommand{\RR}{\mathbb{R}}

\newcommand{\ZZ}{\mathbb{Z}}
\newcommand{\NN}{\mathbb{N}}
\newcommand{\PP}{\mathbb{P}}
\newcommand{\CC}{\mathbb{C}}

\newcommand{\MM}{\mathcal{M}}
\newcommand{\BB}{\mathcal{B}}
\newcommand{\HH}{\mathcal{H}}
\newcommand{\FF}{\mathcal{F}}

\newcolumntype{P}[1]{>{\centering\arraybackslash}p{#1}}
\newcolumntype{C}[1]{>{\centering\arraybackslash}p{#1}}

\usepackage{etoolbox}
\patchcmd{\thebibliography}
  {\settowidth}
  {\setlength{\itemsep}{0pt plus 0.1pt}\settowidth}
  {}{}
\apptocmd{\thebibliography}
  {\small}
  {}{}

\date{}

\title{\textbf{One-dimensional Discrete Models of Maximum Likelihood Degree One}}
\author{Carlos Am\'endola \and Viet Duc Nguyen \and Janike Oldekop}

\begin{document}
\maketitle

\begin{abstract}
We settle a conjecture by Bik and Marigliano stating that the degree of a one-dimensional discrete model with rational maximum likelihood estimator is bounded above by a linear function in the size of its support, 
therefore showing that there are only finitely many fundamental such models for any given number of states. We study these models from a combinatorial perspective with regard to their existence and enumeration. In particular, sharp models, those whose degree attains the maximal bound, enjoy special properties and have been studied as monomial maps between complex unit spheres. In this way, we present a novel link between Cauchy-Riemann geometry and algebraic statistics.
\end{abstract}

\section{From CR Geometry to Algebraic Statistics}

\emph{Cauchy-Riemann or complex-real (CR) geometry} is known as an important branch of complex analysis which arises when studying real submanifolds of complex manifolds~\cite{SeveralComplexVariablesAndTheGeometryOfRealHypersurfaces,ComplexAnalysisAndCRGeometry}. Through its intersections with differential geometry and algebra, it has significant applications in theoretical physics, especially in spacetime geometry, particle physics and string theory. In this article, we build a bridge to \emph{algebraic statistics} \cite{AlgebraicStatistics} via rational mappings between unit spheres. 
The results are valuable to both algebraic statisticians and complex geometers. Another link between rational sphere maps and compressed sensing was proposed in~\cite{RationalSphereMapsLinearProgrammingAndCompressedSensing}.

The fundamental problem in complex analysis we refer to in this article is the study of proper holomorphic functions. Let $\mathbb{B}_N$ be the unit ball in $\CC^N$. Recall that a holomorphic map $F:\mathbb{B}_N \to\mathbb{B}_{n+1}$ is proper if the preimage of every compact set in $\mathbb{B}_{n+1}$ is compact in $\mathbb{B}_N$. 
If $F$ extends continuously to the boundaries,  then $F$ is proper if it maps the boundary of $\mathbb{B}_N$ to the boundary of $\mathbb{B}_{n+1}$. Therefore, we consider mappings $F:\mathbb{B}_N \to\mathbb{B}_{n+1}$ that map the unit sphere in $\CC^N$ to the unit sphere in $\CC^{n+1}$. A widespread interest within CR geometry concerns the complexity of such mappings. A major open problem is to determine, for fixed $N$ and $n$, an upper bound on the degree of $F$. The special case of monomial mappings has been summarized in detail in \cite[Chapter~4]{RationalSphereMaps}. If $F$ is a proper monomial map whose components are the monomials $c_\alpha z^\alpha$, then $\Vert F(z)\Vert^2=1$ whenever $\Vert z\Vert^2=1$, where $\Vert\cdot\Vert$ is the standard Euclidean norm on $\CC^{n+1}$ and $\CC^N$, respectively. By replacing $\vert z_i\vert^2$ with real variables $x_i$, $\Vert F(z)\Vert^2$ becomes a real polynomial in $x_1,\ldots, x_N$ of the same degree as $F$ and $n+1$ non-negative coefficients. Analogously, $\Vert z\Vert^2$ becomes $x_1+\ldots+ x_N$. In particular, $\Vert F(z)\Vert^2$ considered as a polynomial is identically one on the hyperplane given by $x_1+\ldots+ x_N=1$, see also \cite[Section 9]{UniquenessOfCertainPolynomialsConstantOnALine}. By \cite[Proposition 1]{PolynomialProperMapsBetweenBalls}, there is a one-to-one correspondence between such polynomials and proper monomial mappings $F:\mathbb{B}_N \to\mathbb{B}_{n+1}$, up to essential equivalence classes. In order to pass to statistics, restrict the image of such a polynomial map to the non-negative orthant, leading to the probability simplex: $$\Delta_{n} \coloneqq \{ p \in \RR^{n+1} \mid p_0+\ldots + p_n = 1 \text{ and } p_i \ge 0 \text{ for all } i \in [n] \},$$ where $n\in \NN$ and $[n] \coloneqq \{ 0, 1, \ldots, n\}$. In a statistical setting, each element $p\in\Delta_n$ describes a probability distribution of a discrete random variable $X$ with $n+1$ possible states, i.e., $\PP(X=i)=p_i$. From this point of view, a \emph{discrete statistical model} $\MM$ can be considered as a subset of $\Delta_n$. Such models are used to analyze discrete data obtained from a random experiment with $n+1$ outcomes, represented as a vector of counts $u\in\NN^{n+1}$, where $u_i$ is the number of observations of the $i$-th outcome. A general goal in statistics is to study the underlying distribution of such data. To determine a distribution that best explains given data, we often perform \emph{maximum likelihood (ML) estimation}. A necessary prerequisite for this method is the choice of a suitable model. For given data $u$ and a chosen model $\MM$, the \emph{ML estimate} is the distribution $\hat p \in \MM$ that maximizes the \emph{likelihood function} over~$\MM$: $$\hat p = \underset{p\in\MM}{\arg \max} \;\; p_0^{u_0} p_1^{u_1}\ldots p_n^{u_n}.$$ In algebraic statistics, we consider discrete models that can be written as the intersection of $\Delta_n$ and an algebraic variety. The \emph{maximum likelihood degree} of such a model is the number of complex critical points of the likelihood function contained in the associated variety for generic data. A special case are models of ML degree one. Their ML estimator is a rational function of the data \cite{DiscreteStatisticalModelsWithRationalMaximumLikelihoodEstimator, VarietiesWithMaximumLikelihoodDegreeOne}. In this article we focus on models that are described by one-dimensional varieties and have ML degree one. A classification of these models was studied by Bik and Marigliano \cite[Proposition 2.2]{ClassifyingOneDimensionalDiscreteModelsWithMaximumLikelihoodDegreeOne}. With the notion of  \emph{fundamental model} and a variety of combinatorial criteria, they proved their key result that $2n-1$ is an upper bound on the degree of the associated varieties whenever $n\le4$ \cite[Theorem 1.2]{ClassifyingOneDimensionalDiscreteModelsWithMaximumLikelihoodDegreeOne}. In this article we prove \cite[Conjecture 8.1]{ClassifyingOneDimensionalDiscreteModelsWithMaximumLikelihoodDegreeOne}, which extends the statement to arbitrary $n$.

\begin{theorem*}
The degree of a model parameterized by $p: [0,1] \to \Delta_n, \, t \mapsto (c_i t^{\nu_i} (1-t)^{\mu_i})_{i=0}^n$ for suitable non-negative exponents $\nu_i$ and $\mu_i$ and positive real scalings $c_i$ is at most $2n-1$.
\end{theorem*}

Recall that in CR geometry one studies polynomials whose restriction to the line $x+y=1$ is identically equal to one. The parameterization above naturally gives rise to such a polynomial: identifying $t=x$ and $1-t=y$, the exponent vectors $(\nu_i,\mu_i)$ and coefficients $c_i$ define a bivariate polynomial in $\RR[x,y]$ that is constant on the line $x+y=1$. This correspondence via the above theorem allows us to apply tools from CR geometry to the study of one-dimensional models of ML degree one. The main ingredient is the \emph{Newton diagram} \cite[Definition 2.5]{PolynomialsConstantOnAHyperplaneAndCRMapsOfHyperquadrics} (see Section 3), capturing important information about the model. On the other hand, the algebraic statistics setup offers modern approaches that can be used in CR geometry. We correct an error in \cite[Table 2]{UniquenessOfCertainPolynomialsConstantOnALine} and bring attention to fundamental models, a property that has been so far  ignored in the context of monomial sphere~mappings.

\begin{example}
Let $n=2$ and consider the proper monomial map $F: \mathbb{B}_2\to\mathbb{B}_3, \, (z,w) \mapsto (z^3, \sqrt{3}zw, w^3)$, see \cite[Section 2]{PolynomialProperMapsBetweenBalls}. With $x\coloneqq\vert z \vert^2$ and $y\coloneqq\vert w\vert^2$, we derive $f=x^3+3xy+y^3 \in \RR[x,y]$ when considering $\Vert F(z,w)\Vert^2$. Since $f(x,y)=1$ on the line $x+y=1$, the coefficients and exponent vectors of $f$ define a one-dimensional discrete model of ML degree one given by $$p : [0,1] \to \Delta_2, \qquad t \mapsto (t^3, 3t(1-t), (1-t)^3).$$ 
Indeed, its ML estimator for data $u=(u_0,u_1,u_2)$ is given by the rational function 
$$\hat{p}: \Delta_2 \to \mathcal{M}, \quad u \mapsto p(\hat{t}) \quad \text{ where } \quad \hat{t} = \frac{3u_0+u_1}{3u_0+2u_1+3u_2}.$$
Note that the degree of $F$ and the degree of $\mathcal{M}$ are both equal to $3=2\cdot 2-1$.
\end{example}

The paper is organized as follows. In Section \ref{sec:fundamental-models} we introduce fundamental models and their finiteness. Section \ref{sec:labeled-newton-diagrams} is dedicated to a detailed proof of \cite[Conjecture 8.1]{ClassifyingOneDimensionalDiscreteModelsWithMaximumLikelihoodDegreeOne}. In Section \ref{sec:composition-fundamental-models} we construct fundamental models from those of lower dimension. Sharp models, those whose degree is exactly $2n-1$, are discussed in Section \ref{sec:sharp-models}. Section \ref{sec:enumerating-fundamental-models} studies general combinatorial aspects of fundamental models. The paper concludes with a general discussion of possible generalizations to higher-dimensional models of ML degree one. 

\section{On the Finiteness of Fundamental Models} \label{sec:fundamental-models} 

Every one-dimensional model of ML degree one can be represented as the image of a parameterization map of the following form \cite[Proposition 2.2]{ClassifyingOneDimensionalDiscreteModelsWithMaximumLikelihoodDegreeOne}: $$p: [0,1] \to \Delta_n, \quad t \mapsto (c_i t^{\nu_i} (1-t)^{\mu_i})_{i=0}^n,$$ where $\nu_i$ and $\mu_i$ are non-negative exponents and $c_i$ are positive real coefficients. Therefore, the models of interest in this article are completely determined by a finite sequence $(c_i, \nu_i, \mu_i)_{i=0}^n \in (\mathbb{R}_{>0} \times \mathbb{N} \times \mathbb{N})^n$. Following Bik and Marigliano \cite{ClassifyingOneDimensionalDiscreteModelsWithMaximumLikelihoodDegreeOne} we refer to such one-dimensional models of ML degree one as \emph{rational one-dimensional models}, abbreviated as \emph{R1d models}. We call a rational one-dimensional model represented by $(c_i, \nu_i, \mu_i)_{i=0}^n$ \emph{reduced} if the exponent pairs $(\nu_i,\mu_i)$ are pairwise distinct and different from $(0, 0)$. Our main interest is in fundamental models since these are the building blocks for all rational one-dimensional models \cite[Theorem 2.11]{ClassifyingOneDimensionalDiscreteModelsWithMaximumLikelihoodDegreeOne}.

\begin{definition} \label{def:fundamental-model}
A reduced R1d model represented by $(c_i, \nu_i, \mu_i)_{i=0}^n$ is \emph{fundamental} if, given $(\nu_i, \mu_i )$, the scalings $c_i$ are uniquely determined by the constraint $p_0+p_1+\ldots+p_n =1$.
\end{definition}

The degree of any R1d model $\MM\subseteq\Delta_n$ represented by $(c_i, \nu_i, \mu_i)_{i=0}^n$ is $$\deg(\MM) \coloneqq \max \{ \nu_i + \mu_i \mid i \in [n] \}.$$ The following result generalizes \cite[Theorem 1.2]{ClassifyingOneDimensionalDiscreteModelsWithMaximumLikelihoodDegreeOne} to arbitrary $n$.

\begin{theorem} \label{thm:degree-upper-bound}
Let $\MM \subseteq \Delta_n$ be a one-dimensional model of ML degree one. Then $$\deg (\MM) \le 2n-1.$$
\end{theorem}

\begin{proof}[Proof idea]
By \cite[Remark 2.5]{ClassifyingOneDimensionalDiscreteModelsWithMaximumLikelihoodDegreeOne}, for any R1d model $\MM \subseteq \Delta_n$ there exists a reduced model of degree $\deg(\MM)$ in $\Delta_m$ for some $m\leq n$. Therefore, it suffices to prove the upper bound for reduced models. Assume that $\MM$ is reduced and represented by the sequence $(c_i, \nu_i, \mu_i)_{i=0}^n$. 
The degree of $\MM$ corresponds to the total degree of the polynomial $$f \coloneqq c_0 x^{\nu_0} y^{\mu_0} + c_1 x^{\nu_1} y^{\mu_1} + \ldots + c_n x^{\nu_n} y^{\mu_n} \in \RR[x,y].$$ Since $\MM$ is reduced, $f$ has $n+1$ distinct monomials, and each coefficient is non-negative. Furthermore, we have $f(x,y)=1$ on the line $x+y=1$ since this is equivalent to ${f(t,1-t)=1}$. Applying \cite[Theorem 1.1]{ASharpBoundForTheDegreeOfProperMonomialMappingsBetweenBalls} yields the upper bound: $\deg (\MM) \le 2n-1$. 
\end{proof}

We explain the background of \cite[Theorem 1.1]{ASharpBoundForTheDegreeOfProperMonomialMappingsBetweenBalls} and all details in the following section. If $\MM$ is a R1d model represented by $(c_i, \nu_i, \mu_i)_{i=0}^n$, the \emph{support} of $\MM$ is the set of all pairs $(\nu_i, \mu_i)$. Theorem \ref{thm:degree-upper-bound} immediately reveals the finiteness of fundamental models.

\begin{corollary} \label{cor:finiteness-fundamental-models}
For any $n\in \NN$, the number of fundamental models in the probability simplex $\Delta_n$ is finite.
\end{corollary}

\begin{proof}
By Theorem \ref{thm:degree-upper-bound}, the degree of a fundamental model in $\Delta_n$ is at most $2n-1$. This limits their supports to exponent pairs $(\nu_i, \mu_i)$ where $\nu_i + \mu_i \le 2n-1$. Therefore, only a finite number of supports exists. Since fundamental models are uniquely determined by their support, only finitely many fundamental models exist, see also \cite[Remark 2.10]{ClassifyingOneDimensionalDiscreteModelsWithMaximumLikelihoodDegreeOne}.
\end{proof}

We end this section by presenting a fundamental model that is crucial in the detailed proof of Theorem \ref{thm:degree-upper-bound}. The binomial model $\BB_n \subseteq \Delta_n$ is defined by the parameterization map $$p : [0,1] \to \Delta_n, \quad t \mapsto \left(\binom{n}{i} t^i (1-t)^{n-i}\right)_{i=0}^n.$$ It models the number of successes out of an independent sequence of $n$ tries, each one with probability $t \in [0,1]$. It is clearly reduced and has degree $n$. Additionally, we have the following stronger statement.

\begin{theorem} \label{thm:binomial-model-fundamental}
The binomial model $\BB_n$ is fundamental.   
\end{theorem}

\begin{proof}
By definition of a fundamental model we consider the polynomial equation $$\sum_{i=0}^n c_i t^i(1-t)^{n-i} = 1$$ and verify that the scalings $c_0, c_1 \ldots, c_n$ are uniquely determined. By the binomial theorem, $$\sum_{i=0}^n c_i t^i \left (\sum_{k=0}^{n-i} \binom{n-i}{k} (-t)^k \right ) = 1.$$ For $\alpha \in [n]$ we determine the coefficient of $t^\alpha$. Consider those $i\in[n]$ and $k\in[n-i]$ such that $\alpha=i+k$, then $k=\alpha-i$ and $\alpha\ge i$. We obtain the following equality: $$\sum_{i=0}^\alpha c_i t^i \binom{n-i}{\alpha-i} (-t)^{\alpha-i} = \sum_{i = 0}^\alpha c_i \binom{n-i}{\alpha-i} (-1)^{\alpha-i} t^\alpha.$$  Consequently, the linear system to solve is of the form 
\begin{align*}
c_0 &= 1, \quad \textup{ for } \alpha=0,\\
\sum_{i = 0}^\alpha c_i \binom{n-i}{\alpha-i} (-1)^{\alpha-i}&=0, \quad \textup{ for }\alpha>0.    
\end{align*}
 The coefficient matrix is a triangular matrix, i.e., the system has a unique solution. 
\end{proof}

We immediately obtain the following corollary.

\begin{corollary} \label{cor:binomial-model-homogeneous}
Let $\MM\subseteq\Delta_n$ be a reduced model of degree $d$ represented by $(c_i, \nu_i, \mu_i)_{i=0}^n$. If $\nu_i+\mu_i=d$ for all $i \in [n]$, then $\MM$ is unique. In particular, $\MM$ is the binomial model.
\end{corollary}

\begin{remark}
    Corollary \ref{cor:binomial-model-homogeneous} in the language of CR geometry translates to the fact that if the components of a homogeneous polynomial sphere map are linearly independent, then $f$ is a binomial mapping up to unitary transformation. A more general version in several variables with multinomial mappings appears in \cite[Theorem 2.4]{RationalSphereMaps}.
\end{remark}

\section{Reduced Models and Newton Diagrams} \label{sec:labeled-newton-diagrams}

The proof of Theorem \ref{thm:degree-upper-bound} is essentially based on \cite[Theorem 1.1]{ASharpBoundForTheDegreeOfProperMonomialMappingsBetweenBalls}. Due to its combinatorial nature, the proof presented by D'Angelo, Kos and Riehl requires the consideration of a large number of cases, whose full enumeration is left implicit. We present their main concepts in the context of reduced statistical models and give a proof that reduces the number of cases to be considered.  We note the existence of an alternative proof in \cite{SumOfSquaresConjectureTheMonomialCaseInC3} (see their Theorem 4), and yet another proof in \cite[Section 3]{PolynomialsConstantOnAHyperplaneAndCRMapsOfHyperquadrics} with a view towards higher-dimensional mappings, see Section~\ref{sec:outlook} below. The approach presented here is closely linked to the theory explored in~\cite{ClassifyingOneDimensionalDiscreteModelsWithMaximumLikelihoodDegreeOne}.

For any reduced model $\MM\subseteq\Delta_n$ represented by a sequence of the form $(c_i,\nu_i,\mu_i)_{i=0}^n$ we define the following polynomial whose terms are completely described by the model: $$f_\MM \coloneqq c_0x^{\mu_0}y^{\nu_0} + c_1x^{\mu_1}y^{\nu_1} + \ldots + c_n x^{\mu_n}y^{\nu_n} \in \RR[x,y].$$ We further define the set of all such polynomials defined by reduced models as $$\FF \coloneqq \{ f_\MM \in \RR[x,y] \mid \MM \textup{ is a reduced model} \}.$$ 

We record the following useful facts about these polynomials.

\begin{lemma}\label{lemma:homogeneous-polynomial}
Let $\MM \subseteq \Delta_n$ and $f_\MM \in \FF$. Then the following properties hold:
\begin{itemize}
\item[\textup{(i)}] $f_\MM (x,y)=1$ on the line $x+y=1$.
\item[\textup{(ii)}] Each term of $f_\MM$ has a non-negative coefficient.
\item[\textup{(iii)}] If $f_\MM$ is homogeneous of degree $d$, then $f_\MM = (x+y)^d$.
\end{itemize}
\end{lemma}

\begin{proof}
The first two properties follow at once from the definition of a reduced model. For (iii), let $f_\MM\in\FF$ be homogeneous of degree $d$ and suppose that $\MM$ is represented by $(c_i, \nu_i, \mu_i)_{i=0}^n$. Since $f_\MM$ is homogeneous of degree $d$, we have $\nu_i+\mu_i=d$ for all $i\in[n]$, and since $\MM$ is reduced, $\MM$ must be the binomial model $\BB_n$ by Corollary \ref{cor:binomial-model-homogeneous}. Hence, \begin{equation*}f_\MM = \sum_{i=0}^n \binom{n}{i}x^{n-i}y^i=(x+y)^n=(x+y)^d. \qedhere\end{equation*}
\end{proof}

For any reduced model $\MM$ of degree $d$, we decompose $f_\MM$ into its homogeneous parts: $$f_\MM = f_d+f_{d-1}+\ldots+f_1+f_0,$$ where $f_k$ has degree $k$. 

\begin{lemma}[{\cite[Lemma 3.2]{ASharpBoundForTheDegreeOfProperMonomialMappingsBetweenBalls}}] \label{lemma:homogeneous-parts}
Let $\MM \subseteq \Delta_n$ and $f_\MM\in\FF$ of degree $d$. If $f_\MM = \sum_{i\in[d]}f_i$ is the expansion of $f_\MM$ into homogeneous parts, then $$f_d+(x+y)f_{d-1}+\ldots+(x+y)^{d-1}f_1+(x+y)^df_0 = (x+y)^d.$$
\end{lemma}

\begin{proof}
The polynomial on the left-hand side takes the value one on the line $x+y=1$. Since it is homogeneous of degree $d$, it must be $(x+y)^d$ by Lemma~\ref{lemma:homogeneous-polynomial}.
\end{proof}

Multiplying a term $c_{ab}x^ay^b$ of $f_i$ (where $a+b=i$) by $(x+y)$ results in the term $c_{ab}(x^{a+1}y^b+x^ay^{b+1})$ of one degree higher (a `splitting' move), and Lemma~\ref{lemma:homogeneous-parts} shows that iterating this process until all terms have degree $d$ yields $(x+y)^d$. As an important consequence, for any reduced model $\MM$, $f_\MM$ can be constructed from $(x+y)^d$ by a finite sequence of operations 
of the following form: replace $c_{ab}(x^{a+1}y^b+x^ay^{b+1})$ by $c_{ab}x^ay^b$, and keep the other terms the same. This operation is known as an unsplitting move within the theory of \emph{chipsplitting games} \cite[Section 3]{ClassifyingOneDimensionalDiscreteModelsWithMaximumLikelihoodDegreeOne}, and therefore we call this operation~\emph{unsplitting}. 

\begin{remark}
The unsplitting operation defined above corresponds to an `\emph{undoing}' in the CR geometry literature \cite[Remark 2.9 \& Theorem 2.5]{RationalSphereMaps}. This concept has been extensively used to study rational sphere maps and appears in celebrated results in the area such as \cite[Theorem 7.2]{RationalSphereMaps}.
\end{remark}

\begin{definition}
Let $\MM$ and $\MM'$ be two reduced models of degree $d$. Then $\MM$ is an \emph{ancestor} of $\MM'$ if $\MM'$ can be obtained from $\MM$ by finitely many unsplitting moves. 
\end{definition}

The operation of an unsplitting move is shown in the following example.

\begin{example}
The polynomial corresponding to the binomial model in $\Delta_7$ is given by $$f_{\BB_7} = x^7+7x^6y+21x^5y^2+35x^4y^3+35x^3y^4+21x^2y^5+7xy^6+y^7.$$ Replacing $7x^6y+7x^5y^2$ by $7x^5y$ and $7x^2y^5+7xy^6$ by $7xy^5$ yields the following polynomial: $$f_\MM = x^7+7x^5y+14x^5y^2+35x^4y^3+35x^3y^4+14x^2y^5+7xy^5+y^7.$$ That is, $\BB_7$ is an ancestor of the model corresponding to $f_\MM$. One can perform $13$ more unsplitting operations to obtain the following model, which we will return to in Example~\ref{ex:Newton-diagrams-sinks}: $$p:[0,1]\to\Delta_4, \quad t \mapsto \left(t^7, \frac{7}{2}t^5(1-t), \frac{7}{2}t(1-t), \frac{7}{2}t(1-t)^5, (1-t)^7\right).$$
\end{example}

We note the following factorization of $f_\MM-1$, see also \cite[Definition~3.3]{ASharpBoundForTheDegreeOfProperMonomialMappingsBetweenBalls}.

\begin{proposition}\label{prop:factor}
Let $f_\MM\in\FF$ be of degree $d$. Then there exists a polynomial $g_\MM\in\RR[x,y]$ of degree $d-1$ such that $$f_\MM-1=(x+y-1)g_\MM.$$
\end{proposition}

\begin{proof}
Let $f_\MM\in\FF$ be of degree $d$. By Lemma \ref{lemma:homogeneous-polynomial}, we have $f_\MM(x,y)=1$ on the line $x+y=1$. Thus the polynomial $f_\MM-1$ vanishes on the variety $V_\mathbb{C}(x+y-1)$. By Hilbert's Nullstellensatz \cite[\S4.1 Theorem 2]{IdealsVarietiesAndAlgorithms}, this implies that a power of $f_\MM-1$ lies in the principal ideal $\langle x+y-1 \rangle$. Since $x+y-1$  is irreducible over $\RR$, we obtain $f_\MM-1\in\langle x+y-1 \rangle$. Therefore, there exists $g_\MM\in\RR[x,y]$ such that $f_\MM-1=(x+y-1)g_\MM.$ Since $\textup{deg}(f_\MM-1)=d$ and $\textup{deg}(x+y-1)=1$, we necessarily have $\textup{deg}(g_\MM)=d-1$.
\end{proof}

Studying the coefficients of $g_\MM$ for a reduced model $\MM\subseteq\Delta_n$ leads to the following notion. Let $g_{ab}$ be the coefficient of $x^ay^b$ in $g_\MM$. The \emph{Newton diagram} of $g_\MM$ is $$G_\MM:\ZZ^2\to\{ \texttt{0}, \texttt{P}, \texttt{N} \}, \qquad (a,b) \mapsto \begin{cases} \texttt{P}, & \textup{if } g_{ab}>0, \\ \texttt{0}, & \textup{if } g_{ab}=0, \\ \texttt{N}, & \textup{if } g_{ab}<0. \end{cases}$$ If $\MM$ has degree $d$, the Newton diagram $G_\MM$ can be visualized using a grid of size $d+1$. The grid point with coordinates $(a,b)$ represents the monomial $x^ay^b$. If the coefficient $g_{ab}$ of $x^ay^b$ in $g_\MM$ is zero, positive or negative we assign \texttt{0}, \texttt{P} or \texttt{N} to the grid point $(a,b)$, respectively.

\begin{definition}
    For every entry $(a,b)$ in $G_\MM$ consider the subdiagram that consists of the entry itself, the entry just below, and the entry to the left. Then $(a,b)$ is a \emph{sink} if the subdiagram is one of the subdiagrams shown in Figure \ref{fig:sign-pattern-sinks}. 
\end{definition}

\begin{figure}[htb]
\begin{center}
\begin{BVerbatim}[commandchars=\\\{\}]
P \textcolor{cb-green}{N}     P \textcolor{cb-green}{N}     0 \textcolor{cb-green}{N}     0 \textcolor{cb-green}{N}     P \textcolor{cb-green}{0}     0 \textcolor{cb-green}{0}     P \textcolor{cb-green}{0}
  P       0       P       0       P       P       0
\end{BVerbatim}
\end{center} \vspace{-0.5cm}
\caption{All subdiagrams for an entry to be a sink. The sink is the top-right entry.} \label{fig:sign-pattern-sinks}
\end{figure}

By the following lemma, the support size of $\MM$ is at least as large as the number of sinks. 

\begin{lemma}[{\cite[Proposition 3.8]{ASharpBoundForTheDegreeOfProperMonomialMappingsBetweenBalls}}] \label{lemma:sink-positive-coefficient}
    Let $\MM$ be a reduced model with associated Newton diagram $G_\MM$. If $G_\MM$ has a sink at $(a,b)$, then the coefficient of $x^ay^b$ in $f_\MM$ is positive. 
\end{lemma}

\begin{proof}
The coefficient of $x^ay^b$ in $f_\MM$ is the coefficient of $x^ay^b$ in the following polynomial: $$(x+y-1)(g_{(a-1)b}x^{a-1}y^b+g_{a(b-1)}x^ay^{b-1}+g_{ab}x^ay^b).$$ It is easy to check that this coefficient is positive for all possible sink situations. For example, considering the first subdiagram in Figure \ref{fig:sign-pattern-sinks}, we have $g_{(a-1)b}, g_{a(b-1)}>0$ and $g_{ab}<0$. Then the coefficient of $x^ay^b$ in the polynomial above is $g_{(a-1)b} + g_{a(b-1)}+(-1)g_{ab}>0$.
\end{proof}

This immediately gives the following lower bound.
 
\begin{corollary} \label{cor:support-size-sinks-bound}
The support size of $\MM$ is bounded below by the number of sinks in $G_\MM$.
\end{corollary}

Analogously to a sink, define a \emph{source} by reversing all signs in the definition of a sink, i.e., an entry is a source if its corresponding subdiagram is one of the diagrams shown in~Figure~\ref{fig:sign-pattern-sources}. 

\begin{figure}[htb]
\begin{center}
\begin{BVerbatim}[commandchars=\\\{\}]
N \textcolor{cb-red}{P}     N \textcolor{cb-red}{P}     0 \textcolor{cb-red}{P}     0 \textcolor{cb-red}{P}     N \textcolor{cb-red}{0}     0 \textcolor{cb-red}{0}     N \textcolor{cb-red}{0}
  N       0       N       0       N       N       0
\end{BVerbatim}
\end{center} \vspace{-0.5cm}
\caption{All subdiagrams for an entry to be a source. The source is the top-right entry.} \label{fig:sign-pattern-sources}
\end{figure}

\begin{lemma}[{\cite[Proposition 3.8]{ASharpBoundForTheDegreeOfProperMonomialMappingsBetweenBalls}}] \label{lemma:unique-source}
Let $\MM$ be a reduced model with Newton diagram $G_\MM$. Then $G_\MM$ has a unique source, namely at $(0,0)$.
\end{lemma}

\begin{proof}
We have a source at $(0,0)$ to obtain the $-1$ in $f_\MM-1$. Any other source would force the corresponding coordinate of the model parametrization to have a negative coefficient, which can be seen analogously to the procedure in the proof of Lemma \ref{lemma:sink-positive-coefficient}.
\end{proof}

The Newton diagram $G_\MM$ has at least two sinks on the axes.

\begin{lemma}[{\cite[Lemma 3.9]{ASharpBoundForTheDegreeOfProperMonomialMappingsBetweenBalls}}] \label{lemma:sinks-axes}
Let $\MM$ be a reduced model of degree $d$ with associated Newton diagram $G_\MM$. Then there exist $0 < A, B \le d$ such that $g_{a0}>0$ for $0 \le a < A$ and $g_{a0}=0$ for $A \le a \le d$, and $g_{0b}>0$ for $0 \le b < B$ and $g_{0b}=0$ for $B\le b \le d$. In particular, $G_\MM$ has two sinks on the axes.
\end{lemma}

\begin{proof}
We have $g_{00}=1$. Now assume that there exists an $a$ such that $g_{a0}<0$. Then $(a+1,0)$ with maximal such an $a$ would be a source which contradicts Lemma \ref{lemma:unique-source}. Consequently, we have $g_{a0}\ge 0$ for all $0 \le a \le d-1$. Next we claim that, if $g_{a0}=0$, then $g_{a'0}=0$ for all $a'>a$. Again, if this were false, then we would have a source at the point where the first \texttt{P} after a \texttt{0} occurs. Thus, there is an $A$ such that $g_{a0}>0$ for $0 \le a < A$ and $g_{a0}=0$ for $A \le a \le d$. By symmetry the analogous reasoning holds for entries $(0,b)$. Therefore, $(A,0)$ and $(0,B)$ are sinks in the Newton diagram of $\MM$.
\end{proof}

In the following we will give a stronger lower bound for the number of sinks of a Newton diagram. For this, we define the boundary of a Newton diagram, which is motivated by the following observation. If $\MM$ is a reduced model of degree $d$, then $g_\MM$ is a polynomial of degree $d-1$. That is, considering $G_\MM$, all coordinates on the diagonal where $a+b=d$ are assigned $\texttt{0}$. Depending on the model, $G_\MM$ may have many additional zeros. 

The \emph{boundary} of $G_\MM$ is defined to be the collection of $d+1$ points $(a,b)$ with the following properties: 
\begin{itemize}
\item[(B1)] $(a',b')$ is labeled $\texttt{0}$ whenever $a'\ge a$ and $b'\ge b$.
\item[(B2)] In each row and column we choose the minimal $(a,b)$ satisfying (B1).
\end{itemize}

\begin{definition}
    Let $G_\MM$ be a Newton diagram where $(a,b)$ is labeled \texttt{0} or \texttt{N}. Then $(a,b)$ is \emph{connected to the boundary} if there exists a collection of points $L_i$ for $i \in [k]$ such that 
    \begin{itemize}
        \item[(1)] $L_0 = (a,b)$, and if $L_k = (a',b')$, then $a'+b'=d-1$,
        \item[(2)] $L_i$ has label \texttt{0} or \texttt{N} for all $i \in [k]$, 
        \item[(3)] $L_{i+1}=L_i+(1,0)$ or $L_{i+1}=L_i+(0,1)$ for each $i$.
    \end{itemize}
    We refer to such a collection as the \emph{path} from $(a,b)$ to the $(d-1)$-diagonal.
\end{definition}

Recall that every reduced model $\MM$ of degree $d$ can be obtained from $\BB_d$ by a finite sequence of unsplitting moves. Each unsplitting move replaces a term $c_{ab}(x^{a+1}y^b+x^ay^{b+1})$ in $f_\MM$ by $c_{ab}x^ay^b$, yielding a new polynomial $f_{\MM'}$. If $g_\MM$ is defined by $f_\MM-1=(x+y-1)g_\MM$, then 
\begin{align*}
(x+y-1)(g_\MM-c_{ab}x^ay^b) &= (x+y-1)g_\MM-(x+y-1)c_{ab}x^ay^b \\
&= (f_\MM-1)-c_{ab}x^{a+1}y^b-c_{ab}x^ay^{b+1}+c_{ab}x^ay^b \\ 
&= f_{\MM'}-1.
\end{align*}
Thus, $g_{\MM'} = g_\MM-c_{ab}x^ay^b$. In terms of the Newton diagram of $g_\MM$, an unsplitting move affects only the entry at $(a,b)$: its label may change from \texttt{P} to \texttt{0} or \texttt{N}, depending on the magnitude of $c_{ab}$. Conversely, a splitting move corresponds to multiplying a monomial by $x+y$, which may spread a sink on a lower diagonal to two entries on the next diagonal above.

We now prove the main ingredient for the proof of Theorem \ref{thm:degree-upper-bound}. The proof idea essentially follows \cite{ASharpBoundForTheDegreeOfProperMonomialMappingsBetweenBalls}. However, we clarify several details and provide some alternative arguments. 

\begin{proposition}[{\cite[Proposition 3.11]{ASharpBoundForTheDegreeOfProperMonomialMappingsBetweenBalls}}] \label{prop:sinks-lower-bound}
Let $\MM$ be a reduced model of degree $d$ with Newton diagram $G_\MM$. Then $G_\MM$ has at least $2+\lceil \frac{d-1}{2}\rceil$ sinks.
\end{proposition}

\begin{proof}
We first consider the binomial model $\BB_d$ of degree $d$. Its Newton diagram has zeros on the diagonal where $a+b=d$, and $\texttt{P}$ otherwise. It follows that $G_{\BB_d}$ has $d+1$ sinks. If $\MM$ is another reduced model of degree $d$, we obtain $\MM$ from $\BB_d$ by performing finitely many unsplitting moves. Analogously, we obtain $G_\MM$ from $G_{\BB_d}$ by changing finitely many entries from $\texttt{P}$ to $\texttt{N}$ or $\texttt{0}$.

The main idea of the proof is as follows. In the first step we consider an arbitrary reduced model $\MM$ that has some sinks not on the diagonals $a+b=d$ and $a+b=d-1$. We show that there exists an ancestor of $\MM$, denoted~$\MM^*$, with the following properties: $G_{\MM^*}$ has at most as many sinks as $G_\MM$ and all sinks in $G_{\MM^*}$ are on the $d$- and $(d-1)$-diagonal. This step is based on the different ways in which sinks can propagate from the $d$-diagonal into the interior when creating $\MM$ from $\BB_d$. In the second step we show that a Newton diagram $G_{\MM^*}$ with sinks only on the $d$- and $(d-1)$-diagonal has at least $2+\lceil \frac{d-1}{2}\rceil$ sinks.

\textcolor{cb-red}{\textit{Step 1.}} Let $G_\MM$ be a Newton diagram obtained from a reduced model of degree $d$ that has at least one sink $(s_a,s_b)$ not on the $d$- and $(d-1)$-diagonal. A region containing \texttt{0}'s and \texttt{N}'s that is not connected to the boundary contains at least one sink that is not connected to the boundary. These sinks disappear when the region is set to \texttt{P}. Hence, without loss of generality we assume that all sinks of $G_\MM$ are connected to the boundary. When creating $G_\MM$ from $\BB_d$, sinks can only move to the southwest when performing the corresponding unsplitting moves; otherwise we would create a source, contradicting Lemma~\ref{lemma:unique-source}, see Figure \ref{fig:case1}.
\begin{figure}[htb]
\centering
\begin{BVerbatim}[commandchars=\\\{\}]
\textcolor{cb-green}{0}
P \textcolor{cb-green}{0}
P P \textcolor{cb-green}{0}
P 0 \textcolor{cb-red}{P} 0
P N N N 0
P \textcolor{cb-green}{N} N N \textcolor{cb-red}{P} \textcolor{cb-green}{0}
P P P \textcolor{cb-green}{N} N P \textcolor{cb-green}{0}
\textcolor{cb-red}{P} P P P P P P \textcolor{cb-green}{0}
\end{BVerbatim}

\caption{Regions of \texttt{0}'s and \texttt{N}'s expanding to the north or east result in additional sources.}\label{fig:case1}
\end{figure}
We proceed as follows. We start with $G_\MM$ and set \texttt{0}'s and \texttt{N}'s back to \texttt{P} to obtain~$G_{\MM^*}$ with the desired properties. We distinguish two cases: First, all paths from $(s_a,s_b)$ to the $(d-1)$-diagonal have a unique endpoint. Second, the paths have different endpoints. 

\textcolor{cb-red}{\textit{Case 1 -- Propagation from a unique endpoint.}} Let $(s_a,s_b)$ be a sink in $G_\MM$ such that all paths from $(s_a, s_b)$ to the $(d-1)$-diagonal have a unique endpoint $(e_a,e_b)$. For example, this is the case for the two sinks labeled \texttt{N} in Figure \ref{fig:case1}. If $(e_a,e_b)$ is a sink in $G_\MM$, then changing all entries on the paths except $(e_a,e_b)$ back to \texttt{P} reduces the number of sinks by at least one. If $(e_a,e_b)$ is not a sink in $G_\MM$, then after the operation we have at most as many sinks as in $G_\MM$ since $(s_a,s_b)$ is no longer a sink, but $(e_a,e_b)$ becomes a sink. We continue with the second case where the endpoints are not unique.

\textcolor{cb-red}{\textit{Case 2 -- Propagation via zero segments.}} Assume that $(s_a,s_b)$ is a sink in $G_\MM$ such that all possible paths from $(s_a, s_b)$ to the $(d-1)$-diagonal have $k$ different endpoints $(e_{a_i},e_{b_i})$, $k>1$. For example, this is the case for the sink labeled \texttt{N} in Figure \ref{fig:case2}. Then $(e_{a_i},e_{b_i})$ must be labeled \texttt{0} for all $i\in[k-1]$, which follows from two essential observations. First, if an entry $(e_{a_i},e_{b_i})$ on the $(d-1)$-diagonal is labeled \texttt{N}, then $(e_{a_i}+1,e_{b_i}-1)$ and $(e_{a_i}-1,e_{b_i}+1)$ must be labeled \texttt{P}, otherwise $(e_{a_i}+1,e_{b_i})$ and $(e_{a_i},e_{b_i}+1)$ would be sources, contradicting Lemma~\ref{lemma:unique-source}. Second, the $k$ points $(e_{a_i},e_{b_i})$ on the diagonal must be adjacent to each other. Otherwise there would be at least one \texttt{P} that lies within a region of \texttt{N}'s and \texttt{0}'s, which implies an additional source, again contradicting Lemma \ref{lemma:unique-source}.
\begin{figure}[htb]
\centering
\begin{BVerbatim}[commandchars=\\\{\}]
\textcolor{cb-green}{0}
P \textcolor{cb-green}{0}
P P \textcolor{cb-green}{0}
P P P 0
P N N N 0
P N N \textcolor{cb-red}{P} P \textcolor{cb-green}{0}
P \textcolor{cb-green}{N} N N N N \textcolor{cb-green}{0}
\textcolor{cb-red}{P} P P P P P P \textcolor{cb-green}{0}
\end{BVerbatim}

\caption{Regions of \texttt{0}'s and \texttt{N}'s with \texttt{P}'s in its interior result in additional sources.}\label{fig:case2}
\end{figure}

Consequently, we study a region of \texttt{0}'s and \texttt{N}'s spreading towards the southwest from the diagonal segment with entries $(e_{a_i},e_{b_i})$. Denote the region southwest of $(e_{a_i},e_{b_i})_{i\in[k-1]}$ by $\mathcal{R}$. Such a region cannot touch both axes since $G_\MM$ has at least one \texttt{P} on the $(d-1)$-diagonal. If this were the case, we would again have an additional source, contradicting Lemma \ref{lemma:unique-source}.
\begin{figure}[htb]
\centering
\begin{BVerbatim}[commandchars=\\\{\}]
\textcolor{cb-green}{0}
P \textcolor{cb-green}{0}
P P \textcolor{cb-green}{0}
\textcolor{cb-red}{P} P P \textcolor{cb-green}{0}
\textcolor{cb-green}{0} 0 0 0 0
P \textcolor{cb-green}{0} \textcolor{cb-green}{0} 0 0 0
P P P \textcolor{cb-green}{0} 0 0 0
\textcolor{cb-red}{P} P P P \textcolor{cb-green}{0} 0 0 0
\end{BVerbatim}

\caption{Non-positive regions touching both axes result in additional sources.}
\end{figure}

Without loss of generality, we assume that the region does not touch the vertical axis. Then, for each of the $k$ entries $(e_{a_i},e_{b_i})$, there is a segment of adjacent \texttt{0}'s and \texttt{N}'s to its left. We denote the left endpoints of these segments by $(l_{a_i}, e_{b_i})$. Note that $(e_{a_i},e_{b_i})$ and $(l_{a_i}, e_{b_i})$ may coincide for some $i\in[k-1]$. If all $(l_{a_i}, e_{b_i})$ are sinks, we set all entries in $\mathcal{R}$ to \texttt{P}. In this way, the sinks migrate to the $(d-1)$-diagonal, so that all $(e_{a_i},e_{b_i})$ become sinks. The resulting Newton diagram has at most as many sinks as $G_\MM$.

Now assume that there exists an $(l_{a_i}, e_{b_i})$ that is no sink. Since $(l_{a_i}-1, e_{b_i})$ is labeled \texttt{P}, $(l_{a_i}, e_{b_i}-1)$ must be labeled \texttt{N}. There are two cases. First, $(l_{a_i}, e_{b_i}-1)$ is no element of the segment $[(l_{a_{i-1}},e_{b_{i-1}}), (e_{a_{i-1}},e_{b_{i-1}})]$, or second, it is. In the first case, we set all entries in the region with \texttt{0}'s and \texttt{N}'s southwest of $(l_{a_i}, e_{b_i})$ to \texttt{P}. Since this region must contain at least one sink and after the operation $(l_{a_i}, e_{b_i})$ is a sink, the number of sinks is not increased in this step. We can now assume that for all $(l_{a_i}, e_{b_i})$ that are no sinks, $(l_{a_i}, e_{b_i}-1)$ belongs to the segment $[(l_{a_{i-1}},e_{b_{i-1}}), (e_{a_{i-1}},e_{b_{i-1}})]$. In particular, this means that the sinks $(l_{a_i}, e_{b_i})$ and $(l_{a_{i-1}},e_{b_{i-1}})$ are merged into one, namely $(l_{a_{i-1}},e_{b_{i-1}})$ when creating $\MM$ from $\BB_d$. 

\begin{figure}[htb]
\begin{center}
\begin{BVerbatim}[commandchars=\\\{\}]
\textcolor{cb-green}{0}
P \textcolor{cb-green}{0}
P P \textcolor{cb-green}{0}
P \textcolor{cb-green}{0} 0 0
P \textcolor{cb-green}{0} 0 0 0
P P \textcolor{cb-green}{0} 0 0 0
P P P \textcolor{cb-green}{0} 0 0 0
P P P P \textcolor{cb-green}{0} \textcolor{cb-green}{0} \textcolor{cb-green}{0} \textcolor{cb-green}{0}
P P P P P P P P \textcolor{cb-green}{0}
\end{BVerbatim}
\raisebox{13\height}{\scalebox{1}{$\rightarrow$}} \quad
\begin{BVerbatim}[commandchars=\\\{\}]
\textcolor{cb-green}{0}
P \textcolor{cb-green}{0}
P P \textcolor{cb-green}{0}
P \textcolor{cb-green}{0} 0 0
P \textcolor{cb-green}{0} 0 0 0
P P \textcolor{cb-green}{0} 0 0 0
P P P 0 0 0 0
P P P \textcolor{cb-green}{N} 0 \textcolor{cb-green}{0} \textcolor{cb-green}{0} \textcolor{cb-green}{0}
P P P P P P P P \textcolor{cb-green}{0}
\end{BVerbatim}
\raisebox{13\height}{\scalebox{1}{$\rightarrow$}} \quad
\begin{BVerbatim}[commandchars=\\\{\}]
\textcolor{cb-green}{0}
P \textcolor{cb-green}{0}
P P \textcolor{cb-green}{0}
P \textcolor{cb-green}{0} 0 0
P \textcolor{cb-green}{0} 0 0 0
P P \textcolor{cb-green}{0} 0 0 0
P P P 0 0 0 0
P \textcolor{cb-green}{N} N N 0 \textcolor{cb-green}{0} \textcolor{cb-green}{0} \textcolor{cb-green}{0}
P P P P P P P P \textcolor{cb-green}{0}
\end{BVerbatim}
\end{center}
\vspace{-9pt}
\caption{Example of two sinks merging into one. In the first step, sinks at $(3,2)$ and $(4,1)$ merge into one sink at $(3,1)$. Then the sink at $(3,1)$ moves further to (1,1).}
\end{figure}

Such a coalescence can only happen in disjoint pairs, since $(l_{a_{i-1}}+1,e_{b_{i-1}}-1)$ and $(l_{a_{i-1}}-1,e_{b_{i-1}}+1)$ must be labeled \texttt{P}. Let $r$ be the number of pairs for which this happened when $\MM$ was created from $\BB_d$. First, assume that $e_{a_1}>0$, i.e., $\mathcal{R}$ does not touch the horizontal axis. Then, $G_\MM$ has at least $(d+1)-(k-1)+r+(k-2r) = d-r+2$ sinks. Here, we assume that $G_\MM$ has no other \texttt{0}'s or \texttt{N}'s on the $(d-1)$-diagonal except the $(e_{a_i},e_{b_i})$ which is not a restriction, since such regions occur independently of each other. We now change all entries in $\mathcal{R}$ back to \texttt{P}. Additionally we set the entries $(e_{a_i},e_{b_i})$ to \texttt{P}, except $r$ non-adjacent ones, which we set to \texttt{N}. The resulting Newton diagram $G_{\MM^*}$ has $d-r+1$ sinks, i.e., less than $G_\MM$. The procedure works analogously if $e_{a_1}=0$. In this case $G_\MM$ has at least $(d+1)-k+r+(k-2r) = d-r+1$ sinks. 
This completes Step 1.

\textcolor{cb-red}{\textit{Step 2.}} Assume that $G_\MM$ has sinks only on the $d$- and $(d-1)$-diagonals. If we create $G_\MM$ from $Q_{\BB_d}$, the number of sinks can only decrease when two sinks on the $d$-diagonal merge into one sink on the $(d-1)$-diagonal. This happens when $G_\MM$ is obtained by changing an entry on the $(d-1)$-diagonal from \texttt{P} to \texttt{N}. Note that changing an entry from \texttt{P} to \texttt{0} on the $(d-1)$-diagonal cannot decrease the number of sinks. This uses the fact that $G_\MM$ must have at least one \texttt{P} on the $(d-1)$-diagonal. Consequently, in the following we study the case where \texttt{P}'s are replaced by \texttt{N}'s. Let $(a,b)$ be a point at which such an exchange takes place. That is, $(a+1,b)$ and $(a,b+1)$ are sinks in $G_{\BB_d}$, but no longer sinks in $G_\MM$. In addition, the entries at $(a+1,b-1)$ and $(a-1,b-1)$ must be  labeled \texttt{P}, otherwise $(a+1,b)$ and $(a,b+1)$ would be sources. As above, we obtain that merging two sinks on the $d$-diagonal into one can only take place in pairs, and these pairs are each disjoint. 

First, assume that $d$ is odd. If $(d,0)$ and $(0,d)$ are sinks we minimize the number of sinks by changing \texttt{P}'s to \texttt{N}'s at $(d-2,1), (d-4,3), \ldots, (1,d-2)$. The resulting Newton diagram has $\frac{d-1}{2}+2$ sinks. If $(d-1,0)$ and $(0,d-1)$ are sinks we minimize the number of sinks by changing \texttt{P}'s to \texttt{N}'s at $(d-3,2), (d-5,4), \ldots, (2,d-3)$. Note that $(d-1,1)$ and $(1,d-1)$ must be sinks. The resulting Newton diagram has $\frac{d-3}{2}+4 = \frac{d-1}{2}+3$ sinks. If $(d,0)$ and $(0,d-1)$ are sinks we minimize the number of sinks by changing \texttt{P}'s to \texttt{N}'s at $(d-2,1),(d-4,3),\ldots,(3,d-4)$. Then $(1,d-1)$ and $(2,d-2)$ are additional sinks and the remaining Newton diagram has $\frac{d-3}{2}+4 = \frac{d-1}{2}+3$ sinks. The case where $(d-1,0)$ and $(0,d)$ are sinks works analogously by symmetry. It follows that the number of sinks in $G_\MM$ is bounded below by $\frac{d-1}{2}+2$ if $\MM$ is a reduced model of odd degree $d$. 

Now assume that $d$ is even. If $(d,0)$ and $(0,d)$ are sinks, we obtain $\frac{d-2}{2}$ sinks on the $(d-1)$-diagonal and one more sink remains on the $d$-diagonal. The resulting Newton diagram has $\lceil \frac{d-1}{2} \rceil+2$ sinks. If $(d-1,0)$ and $(0,d-1)$ are sinks, then the minimal number of sinks obtained by changing \texttt{P}'s to \texttt{N}'s is $\frac{d-4}{2}+5=\lceil \frac{d-1}{2} \rceil+3$. If $(d,0)$ and $(d-1,0)$ are sinks, the minimal number of sinks obtained by changing \texttt{P}'s to \texttt{N}'s is $\frac{d-2}{2}+3=\lceil \frac{d-1}{2} \rceil+2$. In summary, the number of sinks in the Newton diagram $G_\MM$ is at least $\lceil \frac{d-1}{2} \rceil + 2$.
\end{proof}

We can now conclude the proof of Theorem~\ref{thm:degree-upper-bound} from the following result.

\begin{theorem}
Let $f_\MM\in\FF$ be defined by a reduced model $\MM\subseteq\Delta_n$ of degree $d$. Then $$d\le2n-1.$$
\end{theorem}

\begin{proof}
Let $\MM \subseteq \Delta_n$ be a reduced model of degree $d$. By Corollary \ref{cor:support-size-sinks-bound}, the support size of $\MM$ is at least as large as the number of sinks in $G_\MM$. Then, by Proposition \ref{prop:sinks-lower-bound}, $n+1 \ge 2+\frac{d-1}{2}$, or equivalently, $\deg (\MM) \le 2n-1$.
\end{proof}

The proof strategy of Proposition \ref{prop:sinks-lower-bound} is shown in the following example. 

\begin{example} \label{ex:Newton-diagrams-sinks}
We study the fundamental model in $\Delta_4$ given by the parameterization $$p:[0,1]\to\Delta_4, \quad t \mapsto \left(t^7, \frac{7}{2}t^5(1-t), \frac{7}{2}t(1-t), \frac{7}{2}t(1-t)^5, (1-t)^7\right).$$ The corresponding polynomial $f_\MM$ can be written using the following factor: 
\begin{align*}
g_\MM = 1&+x+y+x^2-\frac{3}{2}xy+y^2+x^3-\frac{1}{2}x^2y-\frac{1}{2}xy^2+y^3 \\
&+x^4+\frac{1}{2}x^3y-x^2y^2+\frac{1}{2}xy^3+y^4 \\
&+x^5+\frac{3}{2}x^4y-\frac{1}{2}x^3y^2-\frac{1}{2}x^2y^3+\frac{3}{2}xy^4+y^5 \\
&+x^6-x^5y+x^4y^2-x^3y^3+x^2y^4-xy^5+y^6.
\end{align*}
The Newton diagram associated to $\MM$ has five sinks, located at $(7,0)$, $(5,1)$, $(1,1)$, $(1,5)$, and $(0,7)$. It is shown in Figure \ref{fig:Newton-diagrams-sinks-example} (a). Figure \ref{fig:Newton-diagrams-sinks-example} (b) shows an ancestor of $G_\MM$, that has the same number of sinks, but only on the diagonals where $a+b=d$ and $a+b=d-1$.
\end{example}

\begin{figure}[htb]
\begin{minipage}{0.49\textwidth}
(a)
\begin{center}
\begin{BVerbatim}[commandchars=\\\{\}]
\textcolor{cb-green}{0}
P 0
P \textcolor{cb-green}{N} 0
P P P 0
P P N N 0
P N N N P 0
P \textcolor{cb-green}{N} N P P \textcolor{cb-green}{N} 0
P P P P P P P \textcolor{cb-green}{0}
\end{BVerbatim}
\end{center}
\end{minipage}
\begin{minipage}{0.49\textwidth}
(b)
\begin{center}
\begin{BVerbatim}[commandchars=\\\{\}]
\textcolor{cb-green}{0}
P 0
P \textcolor{cb-green}{N} 0
P P P 0
P P P \textcolor{cb-green}{N} 0
P P P P P 0
P P P P P \textcolor{cb-green}{N} 0
P P P P P P P \textcolor{cb-green}{0}
\end{BVerbatim}
\end{center}
\end{minipage}
\\
\caption{Newton diagrams associated to the model considered in Example \ref{ex:Newton-diagrams-sinks}. The Newton diagram in (b) is an ancestor of the diagram shown in (a). Sinks are highlighted in green.}\label{fig:Newton-diagrams-sinks-example}
\end{figure}

\section{Composition of Fundamental Models} \label{sec:composition-fundamental-models} 

Every reduced model can be constructed from fundamental models by taking a \emph{composite} of fundamental models \cite[Proposition 2.9]{ClassifyingOneDimensionalDiscreteModelsWithMaximumLikelihoodDegreeOne}. We introduce a similar construction that allows us to construct a fundamental model from other fundamental models. Following Bik and Marigliano \cite{ClassifyingOneDimensionalDiscreteModelsWithMaximumLikelihoodDegreeOne}, we represent any reduced model by a function $h:\ZZ^2\to\RR_{\ge0}$ that sends an exponent pair $(\nu_i,\mu_i)$ to its associated coefficient $c_i$. The support of $\MM$~equals~$\textup{supp}(h)$. 

\begin{definition}
Let $\MM_1$ and $\MM_2$ be reduced models represented by $h_1, h_2:\ZZ^2\to\RR_{\ge0}$ and $(a,b)\in\textup{supp}(h_1)$. The \emph{composition} of $\MM_1$ and $\MM_2$ at $(a,b)$ with parameter $0<\eta <  h_1(a,b)$ is the reduced model represented by $$h_\text{comp}^\eta:\ZZ^2\to\RR_{\ge0}, \, (\nu,\mu) \mapsto \hspace*{-3.5pt}\begin{cases} h_1(\nu,\mu), & \hspace*{-3.5pt}\textup{if } \nu<a \text{ or } \mu<b, \\ h_1(\nu,\mu)-\eta, & \hspace*{-3.5pt}\textup{if } (\nu,\mu)=(a,b), \\ h_1(\nu,\mu)+\eta \cdot h_2(\nu-a,\mu-b), & \hspace*{-3.5pt}\textup{else.}\end{cases}$$
\end{definition}

\begin{lemma} \label{lemma:composition-model}
The composition of reduced models is indeed a reduced model.
\end{lemma}

\begin{proof}
Assume $\MM_1$ is represented by $(c_i, \nu_i,\mu_i)_{i=0}^{n-1}\cup(c_n,a,b)$ and let $f_{\MM_2}$ be the polynomial given by $\MM_2$. The polynomial associated with the composition of $\MM_1$ and $\MM_2$ at $(a,b)$ is $$f_{\text{comp}}^\eta\coloneqq c_0x^{\nu_0}y^{\mu_0}+\ldots+c_{n-1}x^{\nu_{n-1}}y^{\mu_{n-1}}+(c_n-\eta)x^ay^b+\eta x^ay^b f_{\MM_2} \in \RR[x,y].$$ Note that $c_n-\eta \geq 0$ since $(c_n,a,b)$ is part of the sequence representation of $\MM_1$ and by the choice of $\eta$. 
Since $\MM_2$ is a model we have $f_{\MM_2}(x,y)=1$ on the line $x+y=1$. Hence, $f_{\text{comp}}^\eta(x,y)=f_{\MM_1}(x,y)=1$ on the line $x+y=1$. Since $h_\textup{comp}(\nu,\mu)\ge0$ for all $(\nu,\mu)\in\ZZ^2$ the composition is a model. By construction, the composition of $\MM_1$ and $\MM_2$ at $(a,b)$ is reduced. 
\end{proof}

A vivid way to present fundamental models is through \emph{chip configurations}. For a model of degree~$d$, consider an integer grid of size $d+1$. The coordinates on that grid represent the possible supports. Each coordinate is given by an integer entry that corresponds to the associated coefficient. For simplicity, we write a dot for a zero entry. Furthermore, we write $-1$ at the point $(0, 0)$ to indicate that the coordinates of the parametrization add up to one. We illustrate the composition of two models and their chip configurations in the following example.

\begin{example} \label{ex:composition}
Let $\MM_1$ and $\MM_2$ be fundamental models in $\Delta_2$ such that $h_1(3,0)=h_1(0,3)=1$, $h_1(1,1)=3$, and $h_2(2,0)=h_2(0,2)=1$, $h_2(1,1)=2$. Their composition at $(3,0)$ with $\eta=1$ is given by $h_\text{comp}^1(0,3)=h_\text{comp}^1(3,2)=h_\text{comp}^1(5,0)=1$, $h_\text{comp}^1(1,1)=3$ and $h_\text{comp}^1(4,1)=2$. The chip configurations for the three models are shown in Figure \ref{fig:chip-configuration-composition}.
\end{example}

\begin{figure}[htb]
\begin{minipage}{0.32\textwidth}
(a)
\begin{center}
\begin{BVerbatim}


 1
 . .
 . 3 .
-1 . . 1
\end{BVerbatim}
\end{center}
\end{minipage}
\begin{minipage}{0.32\textwidth}
(b)
\begin{center}
\begin{BVerbatim}



 1
 . 2
-1 . 1
\end{BVerbatim}
\end{center}
\end{minipage}
\begin{minipage}{0.32\textwidth}
(c)
\begin{center}
\begin{BVerbatim}
 .
 . .
 1 . .
 . . . 1
 . 3 . . 2
-1 . . . . 1
\end{BVerbatim}
\end{center}
\end{minipage}
\vspace{5pt}
\caption{Chip configurations of fundamental models. In (a) and (b), the chip configurations of $\MM_1$ and $\MM_2$ from Example \ref{ex:composition} are shown, while (c) illustrates their composition.} \label{fig:chip-configuration-composition}
\end{figure}

The following result is crucial for our further discussion on fundamental models. 

\begin{theorem} \label{thm:composition-fundamental-models}
Let $\MM_1$ be a fundamental model in $\Delta_{n_1}$ of degree $d_1$ with $h_1(d_1,0)=1$, and let $\MM_2$ be any fundamental model in $\Delta_{n_2}$ of degree $d_2$. The composition of $\MM_1$ and $\MM_2$ at $(d_1,0)$ with $\eta=1$ is a fundamental model in $\Delta_{n_1+n_2}$ of degree $d_1+d_2$. 
\end{theorem}

\begin{proof}
By construction and Lemma \ref{lemma:composition-model}, the composition of $\MM_1$ and $\MM_2$ is a reduced model in $\Delta_{n_1+n_2}$ of degree $d_1+d_2$. It remains to show that the obtained model is fundamental. Let $\{ (d_1,0) \} \cup \{(\nu_i,\mu_i)\}_{i=0}^{n_1-1}$ be the support of $\MM_1$ and $\{ (\sigma_i,\tau_i) \}_{i=0}^{n_2}$ be the support of $\MM_2$. Since $\MM_1$ is fundamental, its scalings are uniquely determined by the constraint $$c_0t^{\nu_0}(1-t)^{\mu_0}+\ldots+c_{n_1-1}t^{\nu_{n_1-1}}(1-t)^{\mu_{n_1-1}}+c_{n_1}t^{d_1}=1.$$ Considering the expanded form, the coefficients are linear functions in the scalings $c_0,\ldots,c_{n_1}$. The above constraint holds for all $t\in[0,1]$ when the constant term is one and all other coefficients of the expanded form are zero. This results in a system of linear equations of the form $$\begin{bmatrix} \star & \ldots & \star & 0 \\ \vdots & \ddots & \vdots & \vdots \\ \star & \ldots & \star & 0 \\ \star & \ldots & \star & 1 \end{bmatrix} \cdot \begin{bmatrix} c_0 \\ \vdots \\ c_{n_1-1} \\ c_{n_1} \end{bmatrix} = \begin{bmatrix} 1 \\ \vdots \\ 0 \\ 0 \end{bmatrix}.$$ Since $\MM_1$ is fundamental with $h_1(d_1,0)=1$, this system has a unique solution with $c_{n_1}=1$. In particular, there exists a finite sequence $(\mathcal{R}_i)_i$ of elementary row operations such that the last equation is $c_{n_1}=1$. Analogously, the scalings of $\MM_2$ are uniquely determined by $$a_0t^{\sigma_0}(1-t)^{\tau_0}+\ldots+a_{n_2}t^{\sigma_{n_2}}(1-t)^{\tau_{n_2}}=1.$$ To show that the composition of $\MM_1$ and $\MM_2$ is fundamental, we consider the constraint 
\begin{align*}
&c_0t^{\nu_0}(1-t)^{\mu_0}+\ldots+c_{n_1-1}t^{\nu_{n_1-1}}(1-t)^{\mu_{n_1-1}} \\
&+t^{d_1} [a_0t^{\sigma_0}(1-t)^{\tau_0}+\ldots+a_{n_2}t^{\sigma_{n_2}}(1-t)^{\tau_{n_2}}]=1.
\end{align*}
Studying the expanded form yields a system of linear equations of the form $$\begin{bmatrix} \star & \ldots & \star & 0 & \ldots & 0 \\ \vdots & \ddots & \vdots & \vdots & \ddots & \vdots \\ \star & \ldots & \star & 0 & \ldots & 0 \\ \star & \ldots & \star & \star & \ldots & \star \\ 0 & \ldots & 0 & \star & \ldots & \star \\ \vdots & \ddots & \vdots & \vdots & \ddots & \vdots \\ 0 & \ldots & 0 & \star & \ldots & \star \end{bmatrix} \cdot \begin{bmatrix} c_0 \\ \vdots \\ c_{n_1-1} \\ a_0 \\ \vdots \\ a_{n_2} \end{bmatrix} = \begin{bmatrix} 1 \\ \vdots \\ 0 \\ 0 \\ 0 \\ \vdots \\ 0 \end{bmatrix},$$ where the upper left block is given by the first $n_1-1$ columns of the coefficient matrix of $\MM_1$ and the lower right block is the coefficient matrix of $\MM_2$. Performing the elementary row operations $(\mathcal{R}_i)_i$ on the first $d_1+1$ equations yields a system of the form $$\begin{bmatrix} \star & \ldots & \star & 0 & \ldots & 0 \\ \vdots & \ddots & \vdots & \vdots & \ddots & \vdots \\ \star & \ldots & \star & 0 & \ldots & 0 \\ 0 & \ldots & 0 & \star & \ldots & \star \\ 0 & \ldots & 0 & \star & \ldots & \star \\ \vdots & \ddots & \vdots & \vdots & \ddots & \vdots \\ 0 & \ldots & 0 & \star & \ldots & \star \end{bmatrix} \cdot \begin{bmatrix} c_0 \\ \vdots \\ c_{n_1-1} \\ a_0 \\ \vdots \\ a_{n_2} \end{bmatrix} = \begin{bmatrix} \star \\ \vdots \\ \star \\ 1 \\ 0 \\ \vdots \\ 0 \end{bmatrix}.$$ In particular, the last equations are exactly the equations of $\MM_2$. Since $\MM_1$ and $\MM_2$ are fundamental, the system has a unique solution, so that the composition is fundamental.
\end{proof}

\begin{remark}
    As motivated by Theorem~\ref{thm:composition-fundamental-models}, in the remainder of the paper we will only consider the composition of $\MM_1$ and $\MM_2$ in the special case where the composition takes place at $(d,0)\in\textup{supp}(h_1)$ and $\eta=h_1(d,0)=1$. That will be the setup whenever we refer in the following sections to the composition of two models. It would be interesting to explore whether Theorem~\ref{thm:composition-fundamental-models} can be extended to more general compositions.
\end{remark}

\section{Sharp Models} \label{sec:sharp-models}

The composition of fundamental models described in Theorem \ref{thm:composition-fundamental-models} requires the first model to have $(d,0)$ in its support if the model has degree $d$. As we will see, this is the case for all models in $\Delta_n$ whose degree is given by the upper bound $2n-1$. We call these models \emph{sharp}.

\begin{definition}
A reduced model in $\Delta_n$  is \emph{sharp} if its degree is $2n-1$, and \emph{almost sharp} if its degree is $2n-2$. We write \emph{(almost) sharp} when referring to both cases.
\end{definition}

It immediately follows that a sharp model has odd degree, while an almost sharp model has even degree. By \cite[Theorem 0]{DegreeEstimatesForPolynomialsConstantOnAHyperplane}, there always exists a sharp model. A well-known example is the model given by the following parameterization $p:[0,1] \to \Delta_n$, see e.g. \cite[Section~2]{PolynomialProperMapsBetweenBalls}: 
\begin{equation} \label{eq:example-sharp-model}
t \mapsto \left(t^{2n-1}, \left(\frac{2n-1}{2i+1} \binom{n+i-1}{2i} t^{n-i-1}(1-t)^{2i+1}\right)_{i=0}^{n-1}\right).
\end{equation} 

By \cite[Lemma 8.2]{ClassifyingOneDimensionalDiscreteModelsWithMaximumLikelihoodDegreeOne}, this is indeed a reduced model in $\Delta_n$ of degree $2n-1$ since the coordinates add up to one. For completeness, we give the proof that this is a fundamental model, which originates from Bik and Marigliano \cite[Example~2.14]{ClassifyingOneDimensionalDiscreteModelsWithMaximumLikelihoodDegreeOne-arXiv}. 

\begin{lemma} \label{lemma:sharp-model-fundamental}
The model given by the parameterization \eqref{eq:example-sharp-model} is fundamental.
\end{lemma}

\begin{proof}
The proof uses the fact that every reduced model represented by a sequence $(c_i,\nu_i,\mu_i)_{i=0}^n$ is fundamental if and only if the vector subspace $\textup{span}_\RR(\{ t^{\nu_i}(1-t)^{\mu_i} \}_{i\in[n]})$ of $\RR[t]$ has dimension $n+1$ \cite[Definition~2.11]{ClassifyingOneDimensionalDiscreteModelsWithMaximumLikelihoodDegreeOne-arXiv}.

For $c, c_0, c_1, \ldots, c_{n-1} \in \RR$ we define $f(t)$ to be the following polynomial: $$c t^{2n-1} + c_0 t^{n-1} (1-t) + c_1 t^{n-2}(1-t)^3+\ldots+c_{n-2}t(1-t)^{2n-3}+c_{n-1}(1-t)^{2n-1}.$$ We prove that $f(t)=0$ implies $c=c_0=\ldots=c_{n-1}=0$. For $i \in [n-1]$ we define $$g_i(t)\hspace{-1pt}= \hspace{-1pt}c_it^{n-i-1}+c_{i+1}t^{n-i-2}(1-t)^2+\ldots+c_{n-2}t(1-t)^{2(n-i-2)}+c_{n-1}(1-t)^{2(n-i-1)}.$$ Considering $f(1)=0$ yields $c=0$. It follows that $g_0=f/(1-t)$ is the zero polynomial. Next, for $i\in[k-1]$, we have $c_i=g_i(1)=0$ and therefore $g_{i+1}=g_i/(1-t)^2=0$. Finally, we obtain $c_k=g_k(1)=0$.
\end{proof}

Motivated by CR geometry, Lebl and Lichtblau \cite[Section 3]{UniquenessOfCertainPolynomialsConstantOnALine} study in detail valuable properties of the polynomials $f_\MM$ defined by sharp models. In particular, all polynomials defined by sharp models up to degree 17 are listed in \cite[Table 1]{UniquenessOfCertainPolynomialsConstantOnALine}. We summarize their results in the following lemma in the algebraic statistics context. In this way, the complexity of searching for sharp models can be reduced by restricting~possible~supports.

\begin{lemma} \label{lemma:sharp-models-properties}
Let $\mathcal{M}$ be a sharp model of degree $d$. 
\begin{itemize}
\item[\textup{(i)}] The support of $\mathcal{M}$ contains $(d,0)$ and $(0,d)$.
\item[\textup{(ii)}] The support of $\mathcal{M}$ does not contain any other $(a,b)$ with $a+b=d$.
\item[\textup{(iii)}] The support of $\mathcal{M}$ does not contain $(k,0)$, $(0,k)$ for $k\in \{1,\ldots,d-1\}$.
\item[\textup{(iv)}] The support of $\mathcal{M}$ contains at least one element $(a,b)$ with $a+b=d-1$.
\item[\textup{(v)}] The support of $\mathcal{M}$ does not contain elements $(j,d-1-j)$ for even $j$.
\end{itemize}
\end{lemma}

\begin{proof}
A proof for each statement can be found in \cite[Section 3]{UniquenessOfCertainPolynomialsConstantOnALine}.
\end{proof}

We can further derive the following statement from \cite[Section 4]{UniquenessOfCertainPolynomialsConstantOnALine}. This generalizes Corollary \ref{cor:finiteness-fundamental-models} to reduced models of degree $2n-1$ and $2n-2$, respectively.

\begin{proposition}[{\cite[Proposition 4.2]{UniquenessOfCertainPolynomialsConstantOnALine}}] 
No two (almost) sharp models have the same support. In particular, there can be at most finitely many (almost) sharp models.
\end{proposition}

\begin{proof}
The proof is given in \cite{UniquenessOfCertainPolynomialsConstantOnALine} and is based on \cite[Proposition 4.1]{UniquenessOfCertainPolynomialsConstantOnALine}.
\end{proof}

\begin{corollary} \label{cor:sharp-models-fundamental}
Every (almost) sharp reduced model is fundamental. 
\end{corollary}

\begin{proof}
If there are only finitely many reduced models in $\Delta_n$ of degree $d$, the scalings must be uniquely determined by $p_0+p_1+\ldots+p_n=1$.
\end{proof}

\section{Enumerating Fundamental Models} \label{sec:enumerating-fundamental-models}

In this section, we study the number of fundamental models in $\Delta_n$ of degree $d$. By Corollary~\ref{cor:finiteness-fundamental-models}, we already know that there are only finitely many fundamental models for each pair~$(n,d)$. The following proposition describes pairs for which no fundamental model exists.

\begin{proposition}\label{prop:none}
Let $n,d\in\NN^+$. There exists no fundamental model in $\Delta_n$ of degree $d$ if $$d<n \quad \textup{or}\quad d>2n-1.$$
\end{proposition}

\begin{proof}
We first consider the case where $d<n$. For this, we study the expanded form of $$c_0t^{\nu_0}(1-t)^{\mu_0}+c_1t^{\nu_1}(1-t)^{\mu_1}+\ldots+c_nt^{\nu_n}(1-t)^{\mu_n}=1.$$ As in the proof of Theorem \ref{thm:composition-fundamental-models}, the coefficients are linear expressions in $c_0, \ldots, c_n$, and the above constraint holds for all $t\in[0,1]$ if all coefficients except of the constant one are zero. This yields a linear system of $d+1$ equations in $n+1$ variables. If $d<n$ the system has more variables than equations, so that the corresponding model cannot be fundamental. The second case where $d>2n-1$ follows immediately from Theorem \ref{thm:degree-upper-bound}.
\end{proof}

For all other combinations of pairs, at least one fundamental model can be constructed.

\begin{proposition} \label{prop:existence-fundamental-model}
Let $n, d\in\NN^+$. There exists a fundamental model in $\Delta_n$ of degree $d$ if $$n\le d\le 2n-1.$$
\end{proposition}

\begin{proof}
For $n=d$, there exists a fundamental model by Theorem \ref{thm:binomial-model-fundamental}, namely the binomial model. Similarly, for $d=2n-1$, the sharp model given by the parameterization \eqref{eq:example-sharp-model} is fundamental by Lemma \ref{lemma:sharp-model-fundamental}. Now let $n> 2$ be fixed and let $n < d < 2n-1$. Then there exists $k \in \{ 1,\ldots,n-1 \}$ such that $d=n+k$. Let $\MM_1$ be the sharp model given by the parameterization \eqref{eq:example-sharp-model} in $\Delta_{k+1}$ of degree $2k+1$ and let $\MM_2$ be the binomial model in $\Delta_{n-k-1}$ of degree $n-k-1$. By Lemma \ref{lemma:sharp-models-properties}, the support of $\MM_1$ contains $(2k+1,0)$, and its corresponding coefficient is one by construction. Then, by Theorem~\ref{thm:composition-fundamental-models}, the composition of $\MM_1$ and $\MM_2$ is a fundamental model in $\Delta_n$ of degree $d=n+k$.
\end{proof}

The above proof idea can help to provide some first lower bounds on the number of fundamental models when they exist, see Proposition \ref{prop:recursive-formula-degree-2n-2}. Determining the actual number of these models given $n$ and $d$ is computationally difficult. Bik and Marigliano computed the number of fundamental models in $\Delta_n$ for $n\le5$ \cite[Table 1]{ClassifyingOneDimensionalDiscreteModelsWithMaximumLikelihoodDegreeOne}. Independently, Lebl and Lichtblau performed some computations to determine the number of polynomials in $\FF$ of degree $d$ with $n+1$ terms \cite[Section 2]{UniquenessOfCertainPolynomialsConstantOnALine}. The set of all polynomials in $\FF$ of degree $d$ is denoted by $\HH(2,d)$, see e.g. \cite{ComplexityResultsForCRMappingsBetweenSpheres, DegreeEstimatesForPolynomialsConstantOnAHyperplane}. Analogous to Corollary \ref{cor:sharp-models-fundamental}, the number of these polynomials coincides with the number of fundamental models whenever it is finite. The results by Lebl and Lichtblau for polynomials in $\HH(2,d)$ where $d=2n-1$, $d=2n-2$ and $d=2n-3$, respectively, can be found in \cite[Table~2]{UniquenessOfCertainPolynomialsConstantOnALine} for small $n$, with corresponding OEIS sequences A143107, A143108 and A143109 \cite{OEIS}. Surprisingly, their numbers are finite not only for $d=2n-1$ and $d=2n-2$, but also for $d=2n-3$. We return to this observation later. Additionally, Lebl determined the number of polynomials in $\HH(2,d)$ for $n=10$ and $d=19$ to be 24, and for $n=11$ and $d=21$ to be two --- the latter case required more than eight months of computing time  \cite{lebl2013addendum}. As part of this article, we performed further computations which are based on \cite{MScThesis}. In particular, we added a new OEIS sequence A386840 for the number of fundamental models when $d=2n-4$. The code is available at the MathRepo page \cite{MathRepo}. Our results are shown in Table~\ref{table:fundamental-models} and Figure~\ref{fig:number-fundamental-models}. Comparing the results with \cite[Table~2]{UniquenessOfCertainPolynomialsConstantOnALine} shows a discrepancy regarding the value for parameters $n=d=3$. The following proposition corrects \cite[Table~2]{UniquenessOfCertainPolynomialsConstantOnALine}. Our corrected version appears in \cite[A387029]{OEIS}. 

\begin{proposition}[Corrigendum]
The number of polynomials in $\HH(2,3)$ with four distinct terms is twelve, instead of eleven.
\end{proposition}

\begin{proof}
There are twelve fundamental models of degree three with support size four. From these models we obtain twelve polynomials which are listed below up to swapping of variables:
\begin{alignat*}{3}
&x^3+x^2y+xy+y, \qquad &&x^3 + 2x^2y+xy^2+y, \qquad &&x^3+x^2y+2xy+y^2, \\
&x^3+3x^2y+2xy^2+y^2, \qquad &&x^3+3x^2y+3xy^2+y^3, \qquad &&x^2y+xy^2+x+y^2, \\
& \qquad &&2x^2y+x^2+2xy^2+y^2.
\end{alignat*}
It is easy to check that these polynomials satisfy the necessary conditions to be a polynomial in $\HH(2,3)$. All other possible supports were excluded by direct computations.
\end{proof}

\begin{table}[htb]
    \centering
    \scalebox{0.933}{\begin{tabular}{|c|ccccccccccccc|}
    \hline
    \diagbox{$n$}{$d$} & 1 & 2 & 3 & 4 & 5 & 6 & 7 & 8 & 9 & 10 & 11 & 12 & 13 \\
    \hline
    1 & 1 & & & & & & & & & & & & \\
    2 & & 3 & 1 & & & & & & & & & &   \\
    3 & & & 12 & 4 & 2 & & & & & & & & \\
    4 & & & & 82 & 38 & 10 & 4 & & & & & & \\
    5 & & & & & 602 & 254 & 88 & 24 & 2 & & & & \\
    6 & & & & & & 6710 & 2421 & 643 & 198 & 32 & 4 & & \\
    7 & & & & & & & 83906 & 23285 & 6445 & 1442 & 332 & 56 & 8 \\
    \hline
    \end{tabular}}
    \caption{Number of fundamental models of degree \( d \) in the simplex $\Delta_n$. All blank entries are zero by Proposition \ref{prop:none}.} \label{table:fundamental-models}
\end{table}

For polynomials in $\HH(2,d)$ where $d=2n-2$ a recursive formula was conjectured on the associated OEIS page for sequence A143108 \cite{OEIS}. 
We prove that this is a lower bound on the number of fundamental models in $\Delta_n$ of degree $2n-2$. For this, denote by $a_n$ the number of fundamental models in $\Delta_n$ of degree $2n-1$. Then the recursive formula is given in the following proposition. 

\begin{proposition} \label{prop:recursive-formula-degree-2n-2}
For $n\ge3$, the number of fundamental models in $\Delta_n$ of degree $2n-2$ is at least $$2(a_1a_{n-1}+a_2a_{n-2}+\ldots+a_{n-1}a_1).$$
\end{proposition}

\begin{proof}
Let $1\le k\le n-1$ and consider a fundamental model $\MM_1$ in $\Delta_k$ of degree $2k-1$. Furthermore, let $\MM_2$ be a fundamental model in $\Delta_{n-k}$ of degree $2(n-k)-1$. In particular, both $\MM_1$ and $\MM_2$ are sharp. By Lemma \ref{lemma:sharp-models-properties}, $\MM_1$ contains $(2k-1,0)$ in its support, and the corresponding scaling is one \cite[Lemma 3.1]{UniquenessOfCertainPolynomialsConstantOnALine}. Then, by Theorem \ref{thm:composition-fundamental-models}, the composition of $\MM_1$ and $\MM_2$ is a fundamental model in $\Delta_n$ of degree $2n-2$. In this way, we construct a total of $a_1a_{n-1}+a_2a_{n-2}+\ldots+a_{n-1}a_1$ fundamental models in $\Delta_n$ of degree $2n-2$. Considering the parameterization of such a model and substituting $s\coloneqq 1-t$ yields another fundamental model in $\Delta_n$ of degree $2n-2$. The substitution corresponds to a reflection on the main diagonal in the corresponding chip configuration. Overall, we obtain the lower bound.
\end{proof}

Currently, we are not aware of any argument why the expression in Proposition \ref{prop:recursive-formula-degree-2n-2} is also an upper bound. However, based on our computations we state the following conjecture. 

\begin{conjecture}
In Proposition \ref{prop:recursive-formula-degree-2n-2}, equality holds.
\end{conjecture}

\begin{figure}[htb]
\centering
\scalebox{0.88}{\begin{tabular}{C{0.5cm}C{0.5cm}C{0.5cm}C{0.5cm}C{0.5cm}C{0.5cm}C{0.5cm}C{0.5cm}C{0.5cm}C{0.5cm}C{0.5cm}C{0.5cm}C{0.5cm}C{0.5cm}C{0.5cm}C{0.5cm}C{0.5cm}}
&&&&&&&&1&&&&&&&&\\
&&&&&&&1&&3&&&&&&&\\
&&&&&&2&&4&&12&&&&&&\\
&&&&&4&&10&&38&&82&&&&&\\
&&&&2&&24&&88&&254&&602&&&&\\
&&&4&&32&&198&&643&&2421&&6710&&&\\
&&8&&56&&332&&1442&&6445&&23285&&83906&&\\
&\textcolor{cb-green}{4}&&\textcolor{cb-yellow}{96}&&?&&?&&?&&?&&285601&&1279349&\\
\textcolor{cb-green}{2}&&\textcolor{cb-yellow}{112}&&?&&?&&?&&?&&?&&?&&?
\end{tabular}}

\caption{Number of fundamental models in $\Delta_n$ of degree $d$. The $n$-th row represents the models in $\Delta_n$, starting with $n=1$, while the $k$-th diagonal represents models of degree $2n-k$ where the north-to-west diagonals are considered. Each question mark entry is at least one. Entries in green were taken from \cite[Table 2]{UniquenessOfCertainPolynomialsConstantOnALine}, entries in yellow are conjectured values. We have included this table in the OEIS \cite[A386841]{OEIS}.} \label{fig:number-fundamental-models}
\end{figure}

We illustrate the recursive formula with an example.

\begin{example}
For $n=3$ we have $2 (a_1 a_2 + a_2 a_1) = 4$ which is precisely the number of fundamental models in $\Delta_n$ of degree $2n-2$. These four models can be constructed from the previous sharp ones. The unique sharp model in $\Delta_1$ is the binomial model with support $\{(1,0),(0,1)\}$ and coefficients $c_0=c_1=1$. The unique sharp model in $\Delta_2$ is shown in Figure \ref{fig:chip-configuration-composition} (a). By symmetry, taking the composition of all possible arrangements yields four fundamental models in $\Delta_3$ of degree four. Their chip configurations are shown in Figure \ref{fig:construction-fundamental-models-(3,4)}.
\end{example}

\begin{figure}[htb]
\begin{center}
\begin{BVerbatim}
 .              1              .              1
 1 .            . 1            . 1.           . . 
 . . .          . . .          . . .          . 3 .
 . 3 . 1        . 3 . .        1 . 3 .        . . . 1
-1 . . . 1     -1 . . 1 .     -1 . . . 1     -1 1 . . .
\end{BVerbatim}
\end{center} \vspace{-0.25cm}
\caption{Chip configurations of fundamental models in $\Delta_3$ of degree four.}\label{fig:construction-fundamental-models-(3,4)}
\end{figure}

Another conjecture concerns the number of reduced models in $\Delta_n$ of degree $2n-3$. By \cite[Table~2]{UniquenessOfCertainPolynomialsConstantOnALine}, this number is finite for $n=3,4,5,6$, and therefore the corresponding reduced models are fundamental. A natural conjecture is that this is true for all $n\ge3$. Knowing the sharp models in $\Delta_{n-1}$, the idea presented in \cite[Proposition 4.1]{UniquenessOfCertainPolynomialsConstantOnALine} can be used to check whether a one-parameter family of reduced models in $\Delta_n$ of degree $2n-3$ cannot exist. In contrast, for $n\ge4$ we prove that $\Delta_n$ always contains a reduced model of degree $d\ge2n-4$ that is not fundamental.

\begin{proposition} \label{prop:existence-one-parameter-family}
Let $n\ge 4$ and $4 \le k \le n$. Then there exists a one-parameter family of reduced models in $\Delta_n$ of degree $d=2n-k$.
\end{proposition}

\begin{proof}
Let $n\ge 4$ and $d=2n-k$ with $4 \le k \le n$. By Proposition \ref{prop:existence-fundamental-model}, there exists a fundamental model in $\Delta_{n-2}$ of degree $d-1$. As constructed in the proof, such a model $\MM$ can be obtained as the composition of a sharp and a binomial model. For $d=2n-4$ and $d=2n-n$, respectively, $\MM$ is the sharp or the binomial model. Hence, $\MM$ contains $(d-1,0)$ in its support and the corresponding coefficient is one by the procedure in the proof of Theorem~\ref{thm:composition-fundamental-models}. In the following, we assume that $\MM$ is defined by the parameterization $$p:[0,1] \to \Delta_{n-2}, \quad t \mapsto (c_0t^{\nu_0}(1-t)^{\mu_0}, \ldots, c_{n-3}t^{\nu_{n-3}}(1-t)^{\mu_{n-3}}, t^{d-1}).$$ We claim that the modification $p_c':[0,1] \to \Delta_n$ where $$t \mapsto (c_0t^{\nu_0}(1-t)^{\mu_0}, \ldots, c_{n-3}t^{\nu_{n-3}}(1-t)^{\mu_{n-3}}, (1-c)t^{d-1}, ct^{d-1}(1-t), ct^d)$$ defines a reduced model in $\Delta_n$ of degree $d$ for all $c \in (0,1)$. There are no two components in $p_c'(t)$ that have the same support. Furthermore, all coefficients are positive since $c\in(0,1)$. It remains to show that the components add up to one for all $t\in[0,1]$. Since $\MM$ is a model, we have $$c_0t^{\nu_0}(1-t)^{\mu_0} + \ldots + c_{n-3}t^{\nu_{n-3}}(1-t)^{\mu_{n-3}} + t^{d-1}=1.$$ The remaining terms satisfy $-ct^{d-1}+ct^{d-1}(1-t)+ct^d = 0$. Hence, $c \mapsto p_c'$ defines a one-parameter family of reduced models in $\Delta_n$ of degree $d$.
\end{proof}

The proof above is constructive and provides at least one such family. However, in practice, we can find other one-parameter families as the following example shows.

\begin{example}
We consider reduced models in $\Delta_4$ of degree four. If $\MM$ is the unique sharp model in $\Delta_2$, the construction in the proof of Proposition~\ref{prop:existence-one-parameter-family} yields a one-parameter family of reduced models with chip configuration as shown in Figure \ref{fig:pattern-(4,4)-infinite} (a). Analogously, the chip configuration in Figure \ref{fig:pattern-(4,4)-infinite} (b) leads to another family of reduced models in $\Delta_4$ of~degree~four.
\end{example}

\begin{figure}[htb]
\begin{minipage}{0.49\textwidth}
(a)\begin{center}
\begin{BVerbatim}
 .
 * .
 . . .
 . * . *
-1 . . * *
\end{BVerbatim}
\end{center}
\end{minipage}
\begin{minipage}{0.49\textwidth}
(b)\begin{center}
\begin{BVerbatim}
 .
 . *
 . . .
 * . * .
-1 * . . *
\end{BVerbatim}
\end{center}
\end{minipage}
\vspace{5pt}
\caption{Chip configurations for models in $\Delta_4$ of degree four that lead to infinitely many reduced models. A star indicates a non-zero entry.} \label{fig:pattern-(4,4)-infinite}
\end{figure}

Up to this point, many of our constructions are based on the binomial model. We conclude this section with another well-known model in $\Delta_n$ of degree $d=n$. Let $X$ be a discrete random variable with values in $\NN^+$ such that $\PP(X=k)=(1-p)^{k-1}p$. That is, $X$ describes the number of Bernoulli trials until the first success. We define the \emph{finite geometric model} to be the model given by the parameterization \begin{equation} \label{eq:geometric-model-parameterization} p : [0,1] \to \Delta_n, \quad t \mapsto (t, (1-t)t, (1-t)^2t, \ldots, (1-t)^{n-1}t,(1-t)^n).\end{equation} Indeed, the parameterization defines a fundamental model in $\Delta_n$ by the following lemma. 

\begin{lemma}
The parameterization \eqref{eq:geometric-model-parameterization} defines a fundamental model in $\Delta_n$ of degree $d=n$.
\end{lemma}

\begin{proof}
We verify that the components add up to one. For simplicity we set $s\coloneqq1-t$. Then $$(1-s)+s(1-s)+s^2(1-s)+\ldots+s^{n-1}(1-s)+s^n=1.$$ Since the exponent pairs are pairwise distinct, the parameterization defines a reduced model in $\Delta_n$. It remains to show that the coefficients are uniquely determined by $p_0+\ldots+p_n=1$. For this, we consider the expanded form of the following polynomial on the left-hand side: $$c_0(1-s)+c_1s(1-s)+c_2s^2(1-s)+\ldots+c_{n-1}s^{n-1}(1-s)+c_ns^n=1.$$ The constraint holds for all $s\in[0,1]$ when $c_0=1$ and $c_{i+1}-c_i=0$ for all $i\in[n-1]$.
\end{proof}

\section{Outlook} \label{sec:outlook}

The complexity of a discrete statistical model can be measured by its degree, ML degree, support size and the dimension of the associated variety. In this article, we focused on one-dimensional discrete models of ML degree one and investigated their complexity based on the degree and support size. An important open question concerns the generalization to higher-dimensional models of ML degree one. While the parameterization given in the following proposition yields higher-dimensional discrete models with rational ML estimator, it is open whether every discrete model of ML degree one has such a parameterization.

\begin{proposition} \label{prop:model-generalization}
Let $\MM$ be an $r$-dimensional model parameterized by $$p: [0,1]^r \to \Delta_n, \quad (t_1, \ldots, t_r) \mapsto (c_i t_1^{\nu_{1i}} \ldots t_r^{\nu_{ri}} (1-t_1-\ldots-t_r)^{\nu_{(r+1)i}})_{i=0}^n,$$ where $\nu_{\alpha i}$ are non-negative and $c_i$ are positive real scalings. Then $\MM$ has ML degree one.
\end{proposition}

\begin{proof}
Let $u \in \CC^n$ be generic. The ML degree is the number of complex critical points of $$\ell_u(t_1, \ldots, t_r) = \sum_{i=0}^n u_i \log(c_i t_1^{\nu_{1i}} \ldots t_r^{\nu_{ri}} (1-t_1-\ldots-t_r)^{\nu_{(r+1)i}}).$$ Considering the partial derivatives yields a linear system which has a unique solution for generic $u$. More precisely, if $\nu_{\alpha i}\neq 0$ and $\nu_{(r+1)i} \neq 0$, the $i$-th summand of $\partial_{t_\alpha} \ell_u(t_1, \ldots, t_r)$~is $$\frac{u_i (\nu_{(r+1)i}t_\alpha+\nu_{\alpha i}(t_1+\ldots+t_r-1))}{t_\alpha(t_1+\ldots+t_r-1)}.$$ If $\nu_{\alpha i} = 0$ and $\nu_{(r+1)i} \neq 0$, the $i$-th summand of $\partial_{t_\alpha} \ell_u(t_1, \ldots, t_r)$ is $$\frac{u_i \nu_{(r+1)i}}{(t_1+\ldots+t_r-1)} = \frac{u_i \nu_{(r+1)i}t_\alpha}{t_\alpha(t_1+\ldots+t_r-1)}.$$ If $\nu_{\alpha i} \neq 0$ and $\nu_{(r+1)i} = 0$, the $i$-th summand of $\partial_{t_\alpha} \ell_u(t_1, \ldots, t_r)$ is $$\frac{u_i \nu_{\alpha i}}{t_\alpha} = \frac{u_i \nu_{\alpha i}(t_1+\ldots+t_r-1)}{t_\alpha(t_1+\ldots+t_r-1)}.$$ If $\nu_{\alpha i}=0$ and $\nu_{(r+1)i} = 0$, the $i$th summand of the partial derivative $\partial_{t_\alpha} \ell_u(t_1, \ldots, t_r)$ is zero. If we do this for all partial derivatives and multiply each equation $\partial_{t_\alpha} \ell_u(t_1, \ldots, t_r)=0$ by $t_\alpha(t_1+\ldots+t_r-1)$, we get a linear system in $t_1, \ldots, t_r$. Hence, the ML degree of the model is one. 
\end{proof}

Analogous to the one-dimensional case, reduced models can also be considered for higher-dimensional models that are parameterized as in Proposition \ref{prop:model-generalization}. Regarding these models, a natural question that arises in view of Theorem \ref{thm:degree-upper-bound} is the following, see also \cite[Problem~1]{DegreeEstimatesForPolynomialsConstantOnAHyperplane}.

\begin{question}
Let $r\ge2$. For $r$-dimensional reduced models $\MM\subseteq\Delta_n$ parameterized as in Proposition \ref{prop:model-generalization} does there exist a sharp upper bound for $\deg(\MM)$ in terms of $n$ and $r$?
\end{question}

In the one-dimensional case, we took advantage of the fact that the degree of the model coincides with the total degree of the associated polynomial. Unfortunately, this is no longer guaranteed for higher-dimensional models. The strong connection between reduced models and CR geometry only allows the specification of an upper bound on the total degree of the associated polynomials, similar to the one in Theorem \ref{thm:degree-upper-bound}. This bound can be translated from work by Lebl and Peters into the algebraic statistics setting.

\begin{corollary}[{of \cite[Theorem 1.1]{PolynomialsConstantOnAHyperplaneAndCRMapsOfSpheres}}]
Let $r\ge2$ and let $\MM\subseteq\Delta_n$ be an $r$-dimensional reduced model parameterized as in Proposition \ref{prop:model-generalization}. The following inequality holds and is~sharp: $$\max\{\nu_{1i}+\ldots+\nu_{(r+1)i} \mid i\in [n] \}\le\frac{n}{r}.$$
\end{corollary}
 
Examples of polynomials of degree $d$ in $r+1$ variables with $n+1$ different terms of non-negative coefficients that are constant on $x_1+\ldots+x_{r+1}=1$ and for which $rd= n$ are known \cite[Example 4]{DegreeEstimatesForPolynomialsConstantOnAHyperplane}. A natural research direction is to bring the work of \cite{DegreeEstimatesForPolynomialsConstantOnAHyperplane, PolynomialsConstantOnAHyperplaneAndCRMapsOfSpheres} to the algebraic statistics context, thus making further progress using tools from CR geometry.

\section*{Acknowledgments}
We thank John P. D'Angelo for his feedback and helpful comments, as well as an anonymous referee, who suggested how to simplify the proof of Proposition~\ref{prop:factor}.

\bibliographystyle{alpha}
\bibliography{bib.bib}

\vspace{0.25cm}

\noindent{\bf Authors' addresses:}
\smallskip
\small 

\noindent Carlos Am\'{e}ndola,
Technische Universit\"at Berlin
\hfill {\tt amendola@math.tu-berlin.de}

\noindent Viet Duc Nguyen,
Technische Universit\"at Berlin
\hfill {\tt vduc@fastmail.com}

\noindent Janike Oldekop,
Technische Universit\"at Berlin
\hfill {\tt oldekop@math.tu-berlin.de}

\end{document}